\documentclass{amsart}

\newcommand{\set}[1]{\left\{#1\right\}}
\newcommand{\abs}[1]{\lvert#1\rvert}

\newcommand{\R}{\mathbb{R}}
\newcommand{\N}{\mathbb{N}}

\newtheorem{theorem}{Theorem}
\newtheorem{lemma}{Lemma}
\newtheorem{corollary}{Corollary}

\DeclareMathOperator\diam{diam}

\begin{document}

\title[Distances Between Points in Level Sets of Continuous Functions]{Characterizing Distances Between Points in the Level Sets of a Class of Continuous Functions on a Closed Interval}

\author[Y. Luo]{Yuanming Luo}
\author[H. Riely]{Henry Riely}

\address[Y. Luo]{School of Mathematics\\
        Georgia Institute of Technology\\
        Atlanta, GA 30332}
        
\address[H. Riely]{Department of Mathematics\\
        Kennesaw State University\\
        Marietta, GA 30060}

\begin{abstract}
Given a continuous function $f:[a,b]\to\R$ such that $f(a)=f(b)$, we investigate the set of distances $|x-y|$ where $f(x)=f(y)$. In particular, we show that the only distances this set must contain are ones which evenly divide $[a,b]$. Additionally, we show that it must contain at least one third of the interval $[0,b-a]$. Lastly, we explore some higher dimensional generalizations.
\end{abstract}

\maketitle

\section{Introduction}

Imagine waking up on a crisp fall morning and deciding to use the day for a hike. You drive to the trailhead and begin to chart a route. Since you must return to your car, your elevation will be the same at the beginning and the end of your walk. Are there other elevations through which you will pass twice? Clearly there are. If you begin on an ascent, you must descend to return to the trailhead, and if you begin with a decent, you must eventually ascend. If the trail is perfectly flat, then at every moment your elevation is shared by every other moment. This intuition is often given in an introductory calculus course to illustrate the intermediate value theorem.

A natural follow up question: Can anything be said about the time elapsed between two points of equal elevation? For instance, if your hike lasts an hour, we know that there are two instants, separated by an hour, of equal elevation, namely the start and the finish. Need there be two such instances separated by a half hour? The answer, it turns out, is yes. Separated by 25 minutes? No, it's possible to design a hike with no 25 minute time interval leaving you at the same elevation that you started. So what is special about 30 minutes? Can we characterize all such durations? In this paper, we answer this and related questions.\footnote{The contents of this paper are motivated by Exercise 5.4.6. in \cite{abbott}.}

\section{Notation}

Given a closed interval $[a,b] \subset \R$ and a real number $\lambda$, we will use $C_\lambda([a,b])$ to represent the set of continuous functions on $[a,b]$ mapping both endpoints to $\lambda$. More precisely,
$$C_\lambda([a,b]) = \set{f:[a,b]\to \R \;|\; f \text{ is continuous and } f(a)=f(b)=\lambda}$$

\noindent
We will use $C_\R([a,b])$ to refer to functions in any $C_\lambda([a,b])$:
$$C_\R([a,b]) = \bigcup_{\lambda\in\R} C_\lambda([a,b]) = \set{f:[a,b]\to \R \;|\; f \text{ is continuous and } f(a)=f(b)}$$

\noindent
Given a function $f\in C_\R([a,b])$, and a subset $X\subseteq [a,b]$, let

$$D_f(X) = \set{d>0: \abs{x-y}=d \text{ and } f(x)=f(y) \text{ for some } x, y \in X }$$

\noindent
If $X$ is absent as in $D_f$, assume $X=[a,b]$.

$A^\mathrm{o}$, $\overline{A}$, and $\partial A$ will be used to denote the topological interior, closure, and boundary of $A$ respectively, $\mu(A)$ will be used for the Lebesgue measure of $A$.

\section{Main Results}

\begin{theorem} \label{thm1}

Let $f$ be a real-valued, continuous function on the closed interval $[a,b]$ such that $f(a)=f(b)$. Given any $n \in \N$, there exist $x$ and $y$ in $[a,b]$ such that $|x-y| = \frac{b-a}{n}$ and $f(x) = f(y)$.
\end{theorem}

\begin{proof}
We may assume without loss of generality that $[a,b] = [0,1]$. If not, just apply the result to $f(a + (b-a)x)$.

Define $g(x) = f(x + \frac1n)  - f(x)$ and consider the sum
\begin{align}
& g(0)  + g\left(\frac1n\right) + g\left(\frac2n\right) + \dots + g\left(\frac{n-1}{n}\right) \label{thm1line1}\\
& = f\left(\frac1n\right) - f(0) + f\left(\frac2n\right) - f\left(\frac1n\right) + \dots + f(1) - f\left(\frac{n-1}{n}\right)\label{thm1line2} \\
& = f(1) - f(0) = 0 \label{thm1line3}
\end{align}
where (\ref{thm1line3}) follows from (\ref{thm1line2}) due to cancellation.\\

If every term in (\ref{thm1line1}) is 0, then the result follows immediately because $f(\frac{k+1}{n})  = f(\frac{k}{n})$ for $k = 0, 1, \dots, n-1$. If (\ref{thm1line1}) contains one or more nonzero terms, then there must be at least one positive and one negative term in order for the sum to be zero. That is, $g\left(\frac{k_1}{n}\right)<0$ and $g\left(\frac{k_2}{n}\right) > 0$ for some integers $k_1$ and $k_2$ between 0 and $n-1$. Thus, by the intermediate value theorem, $g(c) = 0$ for some $c$ between $\frac{k_1}{n}$ and $\frac{k_2}{n}$ (the continuity of $g$ follows from the continuity of $f$). Therefore, we have $f(c + \frac1n)  - f(c) = 0$.
\end{proof}

Thereom \ref{thm1} provides a partial answer to the question posed in the introduction. If we hike for an hour, there will be two instants, 30 minutes apart, of equal elevation because 30 minutes is half of an hour. The same is true for 20 minutes, 15 minutes, etc. We are not done, however, because we haven't ruled out other durations.

\begin{theorem} \label{thm2}
Given a closed interval $[a,b]$, let $0<d<b-a$. If $d$ is not of the form $\frac{b-a}{n}$, then there exists a continuous function $f:[a,b] \to \R$ with $f(a) = f(b)$ such that $d\notin D_f$.\footnote{The definition of $D_f$ is given in Section 2: Notation.}
\end{theorem}

\begin{proof}
Once again, we can assume without loss of generality that $[a,b] = [0,1]$. First, let $p(x)$ be any continuous $d$-periodic function with $p(0)\neq p(1)$. Note that the existence of such functions hinges on the fact that $d\neq\frac1n$. Next, let $m(x)$ be any strictly monotone continuous function such that $m(0)=0$ and $m(1)=p(0)-p(1)$. We can insist on strict monotonicity since $m(0)=0\neq p(0)-p(1)=m(1)$. Then $p+m$ is continuous as the sum of continuous functions. Furthermore, $(p+m)(0) = p(0) = p(1) + p(0) - p(1) = (p+m)(1)$.

To finish, we must show that $d \notin D_{p+m}$. Indeed, for all $x\in [0,1-d]$, we have
\begin{align*}
(p+m)(x+d) - (p+m)(x) &= p(x+d) - p(x) + m(x+d) -m(x)\\
&= 0 + m(x+d) - m(x) \neq 0
\end{align*}
using the monotonicity of $m$ and the periodicity of $p$.
\end{proof}

Taken together, Theorem \ref{thm1} and Theorem \ref{thm2} tell us that, on a hike that begins and ends at the same point, the only durations we know, a priori, will separate times of equal elevation, must evenly divide that total time of the hike. This is expressed formally in the following corollary:\footnote{The definition of $C_\R([a,b])$ is given in Section 2: Notation.}

\begin{corollary} \label{cor1}
\[
\bigcap_{f\in C_\R([a,b])} D_f = \set{\frac{b-a}{n}: n\in\N}.
\]
\end{corollary}

\begin{proof}
Theorem \ref{thm1} gives one inclusion and Theorem \ref{thm2} gives the other.
\end{proof}

Corollary \ref{cor1} characterizes the distances which are common to all functions in $C_\R([a,b])$. One then might wonder whether this represents a small intersection of large overlapping sets or there is a particular $f\in C_\R([a,b])$ such that $D_f = \set{\frac{b-a}{n}: n\in\N}$. It turns out to be the former. Each $D_f$ is considerably larger than the set of divisors of $b-a$. In fact, each $D_f$ contains at least a third of the numbers between 0 and $b-a$. Before we prove it, we need to develop a series of lemmas about $D_f$.

\begin{lemma} \label{inclusion}
If $A \subseteq B$, then $D_f(A) \subseteq D_f(B)$.
\end{lemma}

\begin{proof}
Assume $d\in D_f(A)$. Then there are points $x,y\in A$ such that $|x-y|=d$ and $f(x)=f(y)$. But $A \subseteq B$, so $x$ and $y$ are also in $B$. Thus, $d\in D_f(B)$.
\end{proof}

\begin{lemma} \label{constant}
Let $f$ be a constant function on a bounded set $A\subset\R$. Assume $A$ has a maximum value $m$. Then $\mu(D_f(A)) \geq \mu(A)$.
\end{lemma}

\begin{proof}
Notice that $D_f(A)$ contains the set $m-A = \set{m-a \;|\; a\in A}$. Therefore $\mu(D_f(A)) \geq \mu(m-A) = \mu(A)$.
\end{proof}

\begin{lemma} \label{1 interval}
Let $f\in C_\lambda([a,b])$ and suppose either $f(x) > \lambda$ for all $a<x<b$ or $f(x) < \lambda$ for all $a<x<b$. Then $D_f([a,b]) = (0,b-a]$.
\end{lemma}

\begin{proof}
We may assume without loss of generality that $f(x) > \lambda$ for all $a<x<b$ since $D_f([a,b]) = D_{-f+2\lambda}([a,b])$. In other words, $D_f([a,b])$ does not change when the graph of $f$ is reflected over the line $y=\lambda.$

It is clear that $b-a\in D_f([a,b])$ since $f(a)=f(b)$, so we will let $d\in(0,b-a)$ and show that $d\in D_f([a,b])$. Define $g(x) = f(x+d)-f(x)$. Note that $g(a)=f(a+d)-f(a) = f(a+d) - \lambda > 0$ because $f(a+d) > \lambda$. Also, $g(b-d)=f(b)-f(b-d) = \lambda -f(b-d) < 0$ because $f(b-d) > \lambda$.

The intermediate value theorem guarantees the existence of a $c\in(a, b-d)$ such that $g(c) = f(c+d)-f(c) = 0$, i.e., $f(c+d)=f(c)$. Therefore $d\in D_f([a, b])$.
\end{proof}

\begin{lemma} \label{2 intervals}
Given any $a_1 < a_2 \leq a_3 < a_4$, define $A=[a_1,a_2]\cup[a_3,a_4]$, and let $f:A\to\R$ be a continuous function such that $f(a_k) = \lambda$ for $1\leq k \leq 4$. Suppose either $f(x) > \lambda$ for all $x\in A^\mathrm{o}$ or $f(x) < \lambda$ for all $x\in A^\mathrm{o}$. If $\max_{[a_1,a_2]}(f) \geq \max_{[a_3, a_4]}(f)$, then $D_f(A) \supseteq [a_3-a_1, a_4-a_1]$.
\end{lemma}

\begin{proof}
As before, we can assume without loss of generality that $f(x) > \lambda$  for all $x\in A^\mathrm{o}$.

It is clear that $a_3-a_1, a_4-a_1\in D_f(A)$ since $f(a_1)=f(a_3)=f(a_4)$, so we will let $d\in(a_3-a_1, a_4-a_1)$ and show that $d\in D_f(A)$. We will do this in three cases, depending on whether $d$ is greater than, less than, or equal to $a_4-a_2$. Define $g(x) = f(x+d)-f(x)$.\\

\noindent \textit{Case 1: $d > a_4-a_2$}.

In this case, we compute $g(a_1) = f(a_1+d)-f(a_1) = f(a_1+d) -\lambda > 0$ and $g(a_4-d) = f(a_4) - f(a_4-d) = \lambda - f(a_4-d) < 0$. Here, we've used that $a_1+d \in (a_3,a_4)$ and $a_4-d \in (a_1, a_2)$ and $f>\lambda$ on these two open intervals. The intermediate value theorem then guarantees a $c \in (a_1,a_4-d)$ such that $g(c) = f(c+d)-f(c) = 0$. Hence $d\in D_f(A)$.\\

\noindent \textit{Case 2: $d < a_4-a_2$}.

In this case, once again we compute $g(a_1) = f(a_1+d)-f(a_1) = f(a_1+d) - \lambda > 0$. This time, however, we observe that $g(t) \leq 0$ for some $t\in (a_1, a_2)$. Otherwise, we would have $f(t+d) > f(t)$ for all $t\in (a_1, a_2)$, contradicting the assumption $\max_{[a_1,a_2]}(f) \geq \max_{[a_3, a_4]}(f)$.\smallskip

If $g(t) = 0$, we have $f(t+d) = f(t)$. If $g(t) < 0$, then the intermediate value theorem gives a $c \in (t,a_2)$ such that $g(c) = f(c+d)-f(c) = 0$. In either case, $d\in D_f(A)$.\\

\noindent \textit{Case 3: $d = a_4-a_2$}.

This case is trivial as $f(a_2) = \lambda = f(a_4) = f(a_2 + d)$.
\end{proof}

\begin{lemma} \label{n intervals}
Given any $a_1<b_1\leq a_2<b_2 \leq \dots \leq a_n < b_n$, define $A=\cup_{k=1}^n [a_k, b_k]$ and let $f:A\to\R$ be a continuous function such that $f(a_k) = f(b_k) = \lambda$ for $1\leq k \leq n$. Suppose either $f(x) > \lambda$ for all $x\in A^\mathrm{o}$ or $f(x) < \lambda$ for all $x\in A^\mathrm{o}$. Then $\mu(D_f(A)) \geq \mu(A)$.
\end{lemma}

\begin{proof}
We will use proof by induction on $n$, the number of intervals.\\

\noindent \textit{Base case (n=1):} \\

The base case is covered by Lemma \ref{1 interval}, which gives us $D_f([a_1, b_1]) = (0, b_1 - a_1]$. Therefore, $\mu(D_f([a_1, b_1])) = \mu([a_1, b_1]) = b_1 - a_1$.\\

\noindent \textit{Induction Step:} \\

Our goal is to prove that $\mu(D_f(\cup_{k=1}^{n+1}[a_k, b_k])) \geq \mu(\cup_{k=1}^{n+1}[a_k, b_k])$. Assume, without loss of generality, that $\max_{[a_1,b_1]}(f) \geq \max_{[a_{n+1}, b_{n+1}]}(f)$. We do not lose generality because $D_f$ is invariant with respect to horizontal reflections, i.e., $D_{f(x)} = D_{f(-x)}$. Then, by Lemma \ref{2 intervals}, $D_f([a_1,b_1] \cup [a_{n+1},b_{n+1}]) \supseteq (a_{n+1}-a_1, b_{n+1}-a_1)$. Combining this fact with Lemma \ref{inclusion} gives
\begin{align*}
D_f(\cup_{k=1}^{n+1}[a_k, b_k]) &\supseteq D_f(\cup_{k=1}^{n}[a_k, b_k]) \cup D_f([a_1, b_2] \cup [a_{n+1}, b_{n+1}])\\
&\supseteq D_f(\cup_{k=1}^{n}[a_k, b_k]) \cup (a_{n+1}-a_1, b_{n+1}-a_1).
\end{align*}

Next, observe that $(a_{n+1}-a_1, b_{n+1}-a_1)$ and $D_f(\cup_{k=1}^{n}[a_k, b_k])$ are disjoint. Indeed, if $d\in D_f(\cup_{k=1}^{n}[a_k, b_k])$, then $d \leq b_n - a_1 \leq a_{n+1}-a_1$.
\noindent Computing the length of both sides and applying the induction hypothesis, we get
\begin{align*}
\mu(D_f(\cup_{k=1}^{n+1}[a_k, b_k])) &\geq \mu(D_f(\cup_{k=1}^{n}[a_k, b_k]) \cup (a_{n+1}-a_1, b_{n+1}-a_1))\\
&= \mu(D_f(\cup_{k=1}^{n}[a_k, b_k])) + \mu((a_{n+1}-a_1, b_{n+1}-a_1))\\
&\geq \mu(\cup_{k=1}^{n}[a_k, b_k]) + \mu((a_{n+1}-a_1, b_{n+1}-a_1))\\
&= \mu(\cup_{k=1}^{n+1}[a_k, b_k])
\end{align*} 
\end{proof}

\begin{lemma} \label{count intervals}
Let $\set{I_n}$ be a countable collection of closed intervals and define $A=\cup_{n=1}^\infty I_n$. Assume that $A$ is bounded and $\set{I_n}$ have disjoint interiors. Let $f$ be a continuous function on $A$ such that $f(x) = \lambda$ on the endpoints of each $I_n$ and either $f(x) > \lambda$ for all $x\in A^\mathrm{o}$ or $f(x) < \lambda$ for all $x\in A^\mathrm{o}$. Then $\mu(D_f(A)) \geq \mu(A)$.
\end{lemma}

\begin{proof}
Fix $\epsilon>0$. Since $A$ is bounded and $\set{I_n}$ have disjoint interiors, we know that $\lim\limits_{n\to\infty}\mu\left(\cup_{k=n}^\infty I_k\right) = 0$. Thus there exists some $N\in\N$ such that $\mu\left(\cup_{k=N}^\infty I_k\right) < \epsilon$. Applying Lemma \ref{n intervals} and Lemma \ref{inclusion} yields
\begin{align*}
\mu(D_f(A)) &\geq \mu\left(D_f\left(\cup_{k=1}^N I_k\right)\right)\\
&\geq \mu\left(\cup_{k=1}^N I_k\right)\\
&= \mu(A) - \mu\left(\cup_{k=N}^\infty I_k\right)\\
&> \mu(A) - \epsilon
\end{align*} 
Therefore, $\mu(D_f(A)) \geq \mu(A)$ because $\epsilon$ was arbitrary.
\end{proof}

With access to these lemmas, we are now prepared to prove that $D_f([a,b])$ must contain at least a third of the distances in $(0,b-a]$.

\begin{theorem} \label{main thm}
If $f \in C_\lambda([a,b])$ then $\mu(D_f) \geq \frac{b-a}{3}$.
\end{theorem}

\begin{proof}
Let $A_>$, $A_<$, and $A_=$ be the subsets of $[a,b]$ on which $f$ is greater than, less than, and equal to $\lambda$ respectively.

$A_>$ and $A_<$ are the preimages of open sets under a continuous function and are thus open. Therefore, each is a countable union of open intervals. Applying Lemma \ref{count intervals} to the closure of each tells us that $\mu(D_f(\overline{A_>})) \geq \mu(\overline{A_>}) = \mu(A_>)$ and $\mu(D_f(\overline{A_<})) \geq \mu(\overline{A_<}) = \mu(A_<)$.\footnote{Dropping the closure doesn't change the length because the union of countably many intervals has a countable boundary.} Applying Lemma \ref{constant} to $A_=$ gives $\mu(D_f(A_=)) \geq \mu(A_=)$. Combining these three inequalities with Lemma \ref{inclusion}, we have
\begin{align*}
\mu(D_f([a,b])) &\geq \max\left(\mu(D_f(\overline{A_>})), \mu(D_f(\overline{A_<})), \mu(D_f(A_=))\right)\\
&\geq \max\left(\mu(A_>), \mu(A_<), \mu(A_=)\right)\\
&\geq \frac{b-a}{3}
\end{align*}
where the last line follows from $\mu(A_>) + \mu(A_<) + \mu(A_=) = b-a$.
\end{proof}

\section{Future Work}

\subsection{Is $\frac{b-a}{3}$ a  minimum?} \label{alternativeDf}

Theorem \ref{main thm} establishes a lower bound on $D_f$ for functions in $C_\lambda([a,b])$. The key was to restrict our attention to $A_> = \set{x : f(x)>\lambda}$ because if $f(x)=f(y)$, then either both $x$ and $y$ are in $A_>$ or neither are. The same holds for $A_<$ and $A_=$. In other words, points in $D_f$ cannot arise due to "interactions" among $A_>$, $A_<$, and $A_=$. With this in mind, the bound in Theorem \ref{main thm} seems tight: simply define a function which is positive on the first third of $[a,b]$, negative on the second third, and zero on the last third. Then each of $A_>$, $A_<$, and $A_=$ should contribute $(0,\frac{b-a}{3}]$ to $D_f$. For example, let
\[
f(x) = \begin{cases} \sin x &\mbox{if } 0 \leq x \leq 2\pi \\
0 &\mbox{if } 2\pi \leq x \leq 3\pi \end{cases}.
\]
The reason this strategy doesn't work is $A_=$. Indeed, $D_f(A_>)=D_f(A_<)=(0,\pi]$. However, $A_= = \set{0,\pi}\cup [2\pi, 3\pi]$ and $D_f(A_=) = (0,3\pi]$. This makes $D_f$ as large as possible due to interactions between the points $0$ and $\pi$ and the interval $[2\pi, 3\pi].$

The trouble with the previous example is the presence of isolated points 0 and $\pi$ in $A_=$. The former is unavoidable, but we can eliminate the latter by making $f$ zero \textit{between} the intervals on which it is positive and negative. Let 
\[
f(x) = \begin{cases} \sin x &\mbox{if } 0 \leq x \leq \pi \\
0 &\mbox{if } \pi \leq x \leq 2\pi \\
-\sin x &\mbox{if } 2\pi \leq x \leq 3\pi \end{cases}.
\]
Now $A_= = \set{0,3\pi}\cup [\pi, 2\pi]$ and $D_f(A_=) = (0,2\pi]$, but $D_f$ is still strictly greater than the bound established in Theorem \ref{main thm}.

Had we defined $D_f$ slightly differently to ignore the endpoints of the domain of $f$, the previous example would prove Theorem \ref{main thm} is sharp. More precisely, if we instead define $D_f = \set{d>0: \abs{x-y}=d \text{ and } f(x)=f(y) \text{ for some } a<x<y<b}$, then $D_f = (0,\pi]$ in the previous example.

However, if we stick to our original definition, is there an $f\in C_\lambda([a,b])$ with $\mu(D_f)=\frac{b-a}{3}$. If not, what is the infimum of $D_f$ over all such $f$?\\

\subsection{Generalizations} 

What does $D_f(X)$ look like when $X$ is not a closed interval? We could broaden the class of functions we look at by defining 
\[
C_\lambda(X) = \set{f:\overline{X}\to \R \;|\; f \text{ is continuous and } f(x) = \lambda \text{ for all } x\in\partial X}.
\]
What is $\bigcap D_f$ over all such $f$ and what is the infimum of $\mu(D_f)$?

We could also explore functions with an $n$-dimesional domain and/or $m$-dimesional codomain.

What does $D_f(X)$ look like when $X$ is $n$-dimesional? The more general definition of $C_\lambda(X)$ proposed above works just fine in this case. For simplicity, we might want to start with cubes or spheres, and slowly relax the constraints on $X$. Additionally, as in Section \ref{alternativeDf} we should amend the definition $D_f$ to ignore the boundary of $X$. Otherwise $D_f = (0, \diam(X)]$ always (unless $X$ is disconnected).

What does $D_f([a,b])$ look like when the codomain of $f$ is $m$-dimesional? If $m>1$ the minimum $D_f([a,b])$ becomes $\set{b-a}$. Consider, for example, $f:[0,2\pi] \to \R^2$ defined by $f:x \mapsto (\cos x, \sin x)$. The only pair of points in $[0,2\pi]$ which get mapped to the same output are $0$ and $2\pi$, so $D_f = \set{2\pi}$.

To construct an interesting generalization, we must then restrict our attention to functions mapping a closed interval to some subset $A \subset \R^m$. If $A$ contains any "loops," the minimum $D_f([a,b])$ becomes $\set{b-a}$, so $A$ should be a one-dimensional "loop-free" set.

Lastly, if the previous questions are settled, perhaps we could define
\[
C_\lambda(X,Y) = \set{f:\overline{X}\to Y \;|\; f \text{ is continuous and } f(x) = \lambda \text{ for all } x\in\partial X}
\]
and classify $D_f$ in terms of $X\subset \R^n$ and $Y\subset \R^m$.


\begin{thebibliography}{1}
\bibitem{abbott} 
Abbott, Stephen, {\em Understanding Analysis,} Undergraduate Texts in Mathematics,
Springer, New York, 2015.
\end{thebibliography}
\end{document}